\documentclass[8pt]{article}

\usepackage[english]{babel}

\usepackage[letterpaper,top=2cm,bottom=2cm,left=2cm,right=2cm,marginparwidth=1.75cm]{geometry}

\usepackage{mathtools}
\mathtoolsset{showonlyrefs=false}
\usepackage{graphicx}
\usepackage{tikz}
\usepackage{pgfplots}
\pgfplotsset{compat=1.16}
\usepackage{mathrsfs}
\usepackage{ulem}
\mathtoolsset{showonlyrefs=true}

\usepackage{amsmath,amsthm,amssymb,graphicx,tikz,url}
\usepackage[linesnumbered,ruled,vlined]{algorithm2e}
\usepackage{booktabs}
\usepackage{multirow}
\usepackage{subcaption}

\newtheorem{The}{Theorem}[section]
\newtheorem{Lemma}{Lemma}[section]

\theoremstyle{definition}

\newtheorem{Remark}{Remark}[section]

\numberwithin{equation}{section}

\makeatletter
\providecommand\@floatboxreset{}%
\newenvironment{breakablealgorithm}
  {%
    \refstepcounter{algocf}%
    \noindent\hrule height.8pt depth0pt \kern2pt%
    \renewcommand{\caption}[2][\relax]{%
      {\raggedright \textbf{Algorithm~\thealgocf} ##2\par}%
      \ifx\relax##1\relax
        \addcontentsline{loa}{algocf}{\protect\numberline{\thealgocf}##2}%
      \else
        \addcontentsline{loa}{algocf}{\protect\numberline{\thealgocf}##1}%
      \fi
      \kern2pt\hrule\kern2pt%
    }%
    \noindent
    \ignorespaces
  }
  {%
    \kern2pt\hrule\par
  }
\makeatother

\title{Stable Computation of Laplacian Eigenfunctions Corresponding to Clustered Eigenvalues}

\author{Ryoki Endo\thanks{Graduate School of Science and Technology, Niigata University, Niigata, Japan
(endo@m.sc.niigata-u.ac.jp).}, Xuefeng Liu\thanks{School of Arts and Sciences, Tokyo Woman's Christian University, Tokyo, Japan (xfliu@cis.twcu.ac.jp).}}

\begin{document}
\date{}
\maketitle

\begin{abstract}
The accurate computation of eigenfunctions corresponding to tightly clustered Laplacian eigenvalues remains an extremely difficult problem. In this paper, using the shape difference quotient of eigenvalues, we propose a stable computation method for the eigenfunctions of clustered eigenvalues caused by domain perturbation. 
\end{abstract}






\section{Introduction}\label{sec1}

Accurate computation of eigenfunctions for differential operators corresponding to tightly clustered eigenvalues remains a difficult task. Indeed, the error of an approximate eigenvector is governed by the size of the gap between its corresponding eigenvalue and the nearest other eigenvalue; see Davis–Kahan’s theorem discussed in \cite{davis1970rotation} and \cite[Sec.~11.7]{parlett1998symmetric} for strongly defined operators and \cite[Sec.~5]{LiuVej2022} for weakly defined operators. In solving practical matrix eigenvalue problems, one has to pay effort to address this instability; see, e.g., the mixed‐precision iterative‐refinement algorithm proposed by Ogita and Aishima \cite{Ogita2019}.

In this paper, we focus on the eigenvalue problems of differential operators and propose a stable algorithm to recover accurate eigenfunctions corresponding to clustered eigenvalues induced by domain perturbations. The following model eigenvalue problem of the Laplace operator on a bounded domain $\Omega$ will be considered, while the analysis can be extended to general differential operators:
\begin{equation}
\label{eq:eig-laplace}
    -\Delta u = \lambda u \quad \text{in }\Omega, 
    \qquad 
    u=0 \quad\text{on }\partial\Omega.
\end{equation}

The perturbation of a domain with symmetry can easily introduce closely clustered eigenvalues. As demonstrated in Example 1, for clustering eigenvalues, a straightforward discretization of the Laplacian eigenvalue problem generally produces completely incorrect approximations to eigenfunctions. As the main feature of our proposed method, it utilizes the first‐order variation of each eigenvalue with respect to the domain perturbation, thereby enabling the reconstruction of accurate eigenfunctions. Assuming that first‐order variations of eigenvalues remain separated, the proposed algorithm for eigenfunction computation is robust to both discretization and rounding errors, even when the perturbated eigenvalues are arbitrarily close.

\smallskip

\noindent\textbf{Example 1 (Failure case of eigenfunction computation).} 
Let us consider the computation of the second and third Dirichlet eigenfunctions of \eqref{eq:eig-laplace} defined on a rectangular domain $\Omega_\varepsilon := (0,1+\varepsilon)\times(0,1)$. The second and third eigenvalues are
\[
\lambda_2 \;=\; \frac{4\pi^2}{(1+\varepsilon)^2}=4\pi^2(1+O(\varepsilon)), 
\qquad
\lambda_3 \;=\; 4\pi^2.
\]
For a small value $\varepsilon = 10^{-4}$, the closely spaced eigenvalues $\lambda_2$ and $\lambda_3$ will introduce failure in the eigenfunction computation.

Figure~\ref{fig:analytical} displays the analytically computed eigenfunctions, where the second eigenfunction (left) is antisymmetric about the vertical line $x = (1 + \varepsilon)/2$ and the third eigenfunction (right) is antisymmetric about the horizontal line $y = 1/2$.
\begin{figure}[h]
    \centering
    \includegraphics[width=0.6\linewidth]{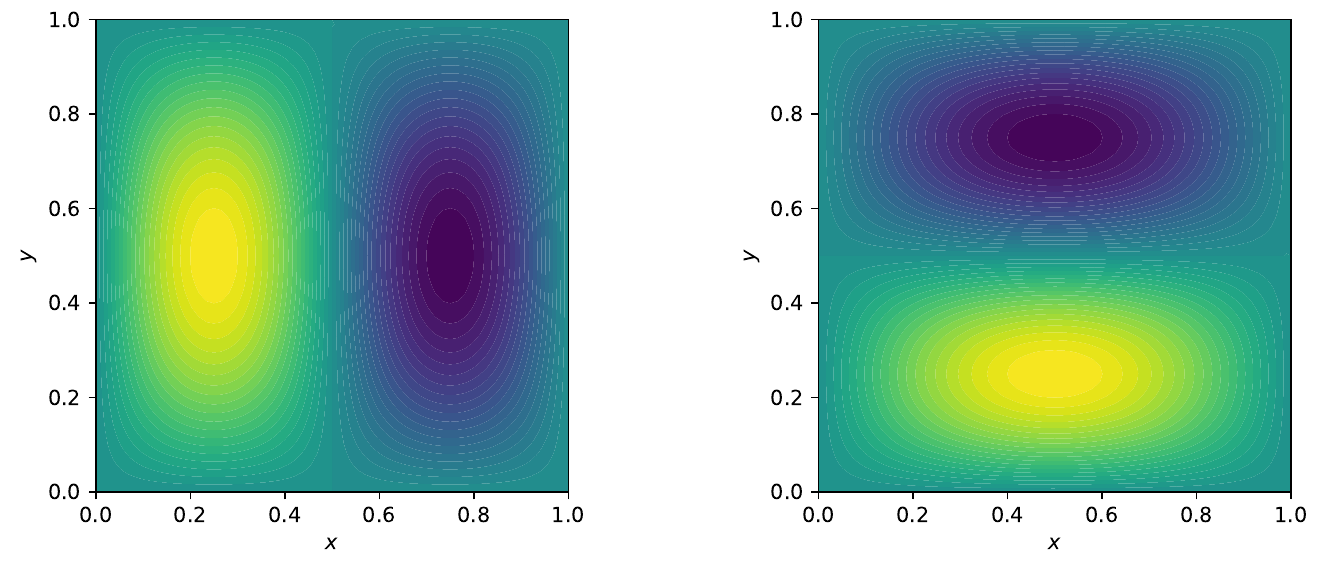}
    \caption{Analytically computed eigenfunctions on $\Omega_\varepsilon$ ($\varepsilon=10^{-4}$).}
    \label{fig:analytical}
\end{figure}



For clustered eigenvalues, a straightforward discretization of the Laplacian eigenvalue problem generally produces completely incorrect approximations to eigenfunctions. This instability is highlighted in Figure \ref{fig:fem_mesh_diagonal_comparison}, which demonstrates how numerically computed results can be highly sensitive to the mesh configuration.  For a simple rectangular domain, different triangulation methods—using "left" or "right" diagonals, or a symmetric "crossed" pattern (see Figure \ref{fig:fem_3_meshes}) —produce visibly inconsistent and distorted eigenfunctions. This demonstrates the fragility of standard methods, where results depend heavily on the mesh configuration.

\begin{figure}[h]
    \centering
    \includegraphics[width=0.65\textwidth]{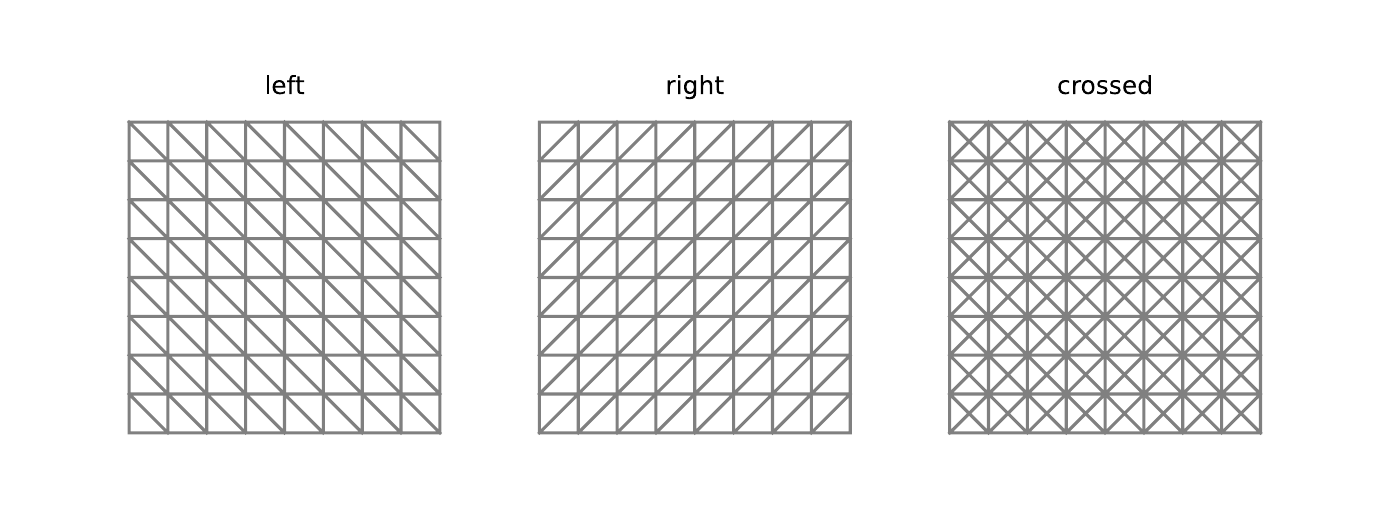}
    \caption{
        Three kinds of meshes for the rectangular domain.
    }
    \label{fig:fem_3_meshes}
\end{figure}

The configuration of this computation was:
\[
\text{Linear Lagrange FEM},\quad
\text{mesh size: }h=1/64,\quad
\#\text{DOF}=8\,192.
\]

\begin{figure}[h]
    \centering
    \includegraphics[width=0.75\textwidth]{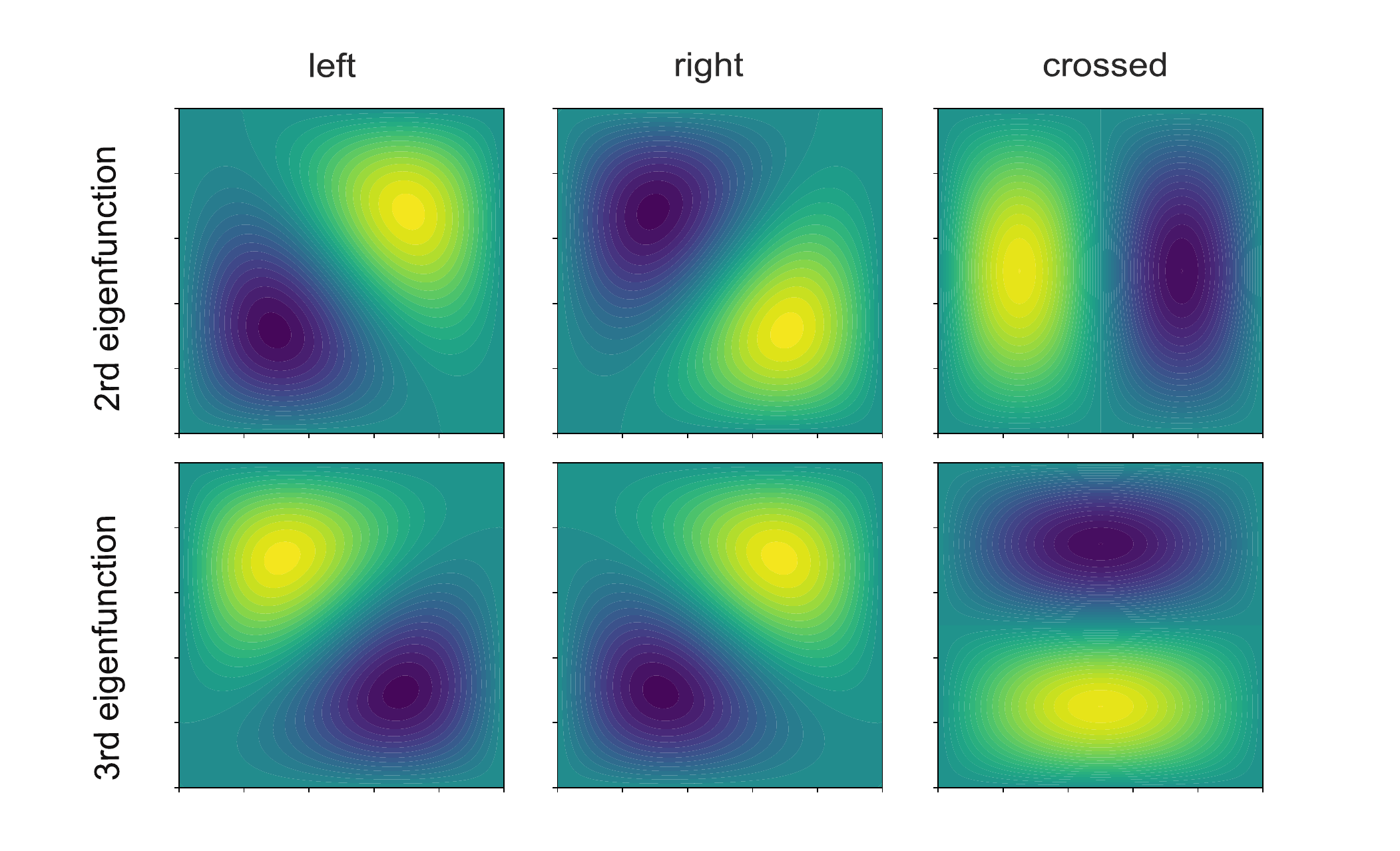}
    \caption{
        Comparison of the 2nd (top row) and 3rd (bottom row) eigenfunctions of the Laplacian, numerically computed on a perturbed rectangle $\Omega_\varepsilon$ with $\varepsilon = 10^{-4}$.
        Each column corresponds to a different triangulation method for the finite element mesh: \texttt{left}, \texttt{right}, and \texttt{crossed}.
    }
    \label{fig:fem_mesh_diagonal_comparison}
\end{figure}

This paper utilizes the shape difference quotient, informed by Rousselet’s directional derivative formula for multiple eigenvalues \cite{Rousselet}, to capture the variation of each eigenvalue corresponding to the domain perturbation. Note that the calculation of the shape difference quotient reduces to solving a small‐scale eigenvalue problem (see Theorem~\ref{lem:Fte-basis-eigen}), which is stable to obtain eigenvectors if the shape derivatives are well separated. Therefore, the use of a shape derivative allows us to transform the ill‐posed eigenfunction computation problem into a well‐posed one. As will be demonstrated in Section~\ref{sec:numerical_examples_rec}, even for an extremely small value $\varepsilon = 10^{-10}$, the proposed algorithm provides correct approximations to the target eigenfunctions.

\medskip

This paper is organized as follows. Section \ref{section:preliminary} provides preliminary definitions and the problem setting. In Section \ref{section:main-theories}, we derive the main theory for our proposed method, including the formulation via matrix eigenvalue problems. Section \ref{section:algorithm} presents the algorithm for the stabilized computation of Laplacian eigenfunctions based on these theories. We then demonstrate the effectiveness of our approach through two numerical examples: on a perturbated rectangular domain in Section \ref{sec:numerical_examples_rec}  and on a perturbated equilateral triangle in Section 6. Finally, Section \ref{sec:conclusion} offers concluding remarks and discusses future work.

\section{Preliminary}\label{section:preliminary}

Let $K \subset \mathbb{R}^2$ be a polygonal domain. We denote by $L^2(K)$ the space of real‐valued functions on $K$ that are square‐integrable. The Sobolev space $H^1(K)$ consists of those functions in $L^2(K)$ whose weak first‐order partial derivatives also lie in $L^2(K)$. Finally, we write $H^1_0(K)$ for the closed subspace of $H^1(K)$ enforcing homogeneous Dirichlet boundary conditions.

We write $\|v\|_{K}$ for the $L^2$‐norm of any $v\in L^2(K)$, and denote by $(\cdot,\cdot)_K$ the inner product on $L^2(K)$ or $(L^2(T))^2$. Since functions in $H^1_0(K)$ vanish on $\partial K$, the bilinear form 
$(\nabla \cdot, \nabla \cdot)_K$
defines an inner product on $H^1_0(K)$. 
With this notation, the variational formulation of the Laplacian eigenvalue problem \eqref{eq:eig-laplace} reads as
\begin{align}
\text{Find }u \in H^1_0(K)\setminus\{0\}\text{ and }\lambda > 0 \text{ such that }
(\nabla u,\nabla v)_K = \lambda\, (u,v)_K \quad \forall\,v \in H^1_0(K).
\label{eq:eigenvalue-problem}
\end{align}
Since the inverse Laplacian is a compact, self‐adjoint operator on $L^2(K)$, the spectral theorem guarantees that \eqref{eq:eigenvalue-problem} admits a sequence of eigenvalues
\[
0 < \lambda_1(K) < \lambda_2(K) \le \lambda_3(K) \le \cdots,
\]
each repeated according to its multiplicity, and an orthonormal basis of eigenfunctions in $L^2(K)$.

\section{Main theories}\label{section:main-theories}

Let $p = (x_1, x_2, \ldots, x_k, y_1, y_2, \ldots, y_k) \in \mathbb{R}^{2k}$, and denote by $K^p$ the $k$‐sided polygon whose vertices are given by $(x_i,y_i)$ for $i=1,\ldots,k$. For $t\ge 0$ and an $\ell^2$‐normalized vector $e\in\mathbb{R}^{2k}$, set $p_t := p + t\,e$. Thus $K^t := K^{p_t}$ is a perturbated polygonal domain for small $t\geq 0$ such that the edges of the polygon are not accrossing with each other.  Let 
\[
K^{t}_h \;:=\; T_{1,t} \cup T_{2,t} \cup \cdots \cup T_{m,t}
\]
be a triangulation of $K^t$ by disjoint triangles such that 
\[
i \neq j \quad\Longrightarrow\quad \mathrm{int}(T_{i,t}) \cap \mathrm{int}(T_{j,t}) = \emptyset.
\]
For each triangle $T_{j,0}\subset K^{0}(=K^p)$, let $S_{j,t}$ be the affine map, represented by a $2\times 2 $ matrix, carrying $T_{j,0}$ onto $T_{j,t}$ for $j=1,\dots,m$.  Then, one defines the global piecewise‐affine map 
\[
\Phi_t \;:\; K^0 \;\longrightarrow\; K^t, 
\qquad 
\Phi_t\big|_{T_{j,0}} = S_{j,t}\big|_{T_{j,0}},\; j=1,\dots,m.
\]

\medskip

Write $\lambda_i^t$ for the $i$‐th Dirichlet eigenvalue of the Laplacian on $K^t$.  Define the corresponding difference quotient at $p$ by
\begin{equation}\label{def:nabla-e-t}
D_t \lambda_i \;:=\; \frac{\lambda_i^t - \lambda_i^0}{t}.
\end{equation}
Through the paper, it is assumed that, for triangle domain $K$ with  parameter $p$, the eigenvalues $\lambda_n^p, \ldots, \lambda_N^p$ coincide:
\begin{equation}
\label{eq:basic_assumption_on_eigenvalue}
\lambda_n^p = \lambda_{n+1}^p = \cdots = \lambda_N^p \;=\; \lambda,    
\end{equation}
and let 
\[
E \;:=\; \mathrm{span}\{\,u_n^0,\dots,u_N^0\}
\]
denote the eigenspace on $K^0$ corresponding to $\lambda$.

For each $i=n,\dots,N$, choose an eigenpair $(\lambda_i^t,\,u_i^t)$ on $K^t$:
\[
-\,\Delta\,u_i^t \;=\; \lambda_i^t\,u_i^t \quad\text{in }K^t,\qquad
u_i^t\big|_{\partial K^t} = 0,
\]
and assume that $\{\,u_n^t,\dots,u_N^t\}$ are linearly independent.  Define the pull‐back
\[
\widetilde u_i^t \;:=\; u_i^t \circ \Phi_t 
\;\in\; H^1_0(K^0),
\qquad i = n,\dots,N.
\]
On each triangle $T_{j,0}\subset K^0$, one has
\[
\nabla\bigl(u_i^t\big|_{T_{j,t}}\bigr)
\;=\; S_{j,t}^{-\mathsf{T}} \;\nabla\bigl(\widetilde u_i^t\big|_{T_{j,0}}\bigr),
\quad j=1,\dots,m.
\]
Define, for each triangle index $j=1,\dots,m$,
\[
P_{j,t} \;:=\; \frac{S_{j,t}^{-1} S_{j,t}^{-\mathsf{T}} - I}{t}, 
\qquad 
d_{j,t} \;:=\; \frac{\lvert\det S_{j,t}\rvert - 1}{t}.
\]
Then introduce two bilinear forms on $H^1_0(K^0)$:

\begin{align}
\label{eq:def-tilde_at}
\widetilde a_t(u,v) 
&:= \sum_{j=1}^m \Bigl[ 
\lvert\det S_{j,t}\rvert\cdot\bigl(P_{j,t}\,\nabla u,\;\nabla v\bigr)_{L^2(T_{j,0})}\, 
\;\\
&~~+\; 
d_{j,t}\cdot\bigl(\nabla u,\nabla v\bigr)_{L^2(T_{j,0})} 
\;-\; 
\lambda\, d_{j,t} \cdot \bigl(u,v\bigr)_{L^2(T_{j,0})}\
\Bigr], \notag
\end{align}

\begin{align}
\widetilde b_t(u,v) 
:= \sum_{j=1}^m d_{j,t}\cdot\bigl(u,v\bigr)_{L^2(T_{j,0})}. 
\label{eq:def-tilde_bt}
\end{align}

\begin{Lemma}\label{lem:derivative-esimation-close-eigenvalues-prep-poly}
For each $i=n,\dots,N$, the pair $\bigl(D_t\lambda_i,\;\widetilde u_i^t\bigr)$ satisfies
\begin{equation}\label{eq:derivative-esimation-close-eigenvalues-prep-poly}
\widetilde a_t\bigl(\widetilde u_i^t,\;w\bigr)
\;=\;
D_t\lambda_i\;\widetilde b_t\bigl(\widetilde u_i^t,\;w\bigr)
\quad
\forall\,w\in E.
\end{equation}
In other words, $(D_t\lambda_i,\widetilde u_i^t)$ ($i=n,\cdots,N$) are the eigenpairs of the finite‐dimensional problem \eqref{eq:derivative-esimation-close-eigenvalues-prep-poly}.
\end{Lemma}

\begin{proof}
On the perturbated domain $K^t$, each $u_i^t$ satisfies 
\[
\int_{K^t} \nabla u_i^t\cdot \nabla v\,dx 
\;=\; 
\lambda_i^t \int_{K^t} u_i^t\,v\,dx
\quad 
\forall\,v\in H^1_0(K^t).
\]
Pulling back via $v = w\circ \Phi_t^{-1} \in H_0^1(K^t)$, with $w\in H^1_0(K^0)$, and using the piecewise‐affine map $\Phi_t$, one obtains for each fixed $i$:
\begin{equation}\label{eq:perturbated-substitute-u-Kt-v}
    \sum_{j=1}^m 
\int_{T_{j,0}}
\bigl(S_{j,t}^{-\mathsf{T}} \,\nabla \widetilde u_i^t \bigr)\cdot 
\bigl(S_{j,t}^{-\mathsf{T}} \,\nabla w \bigr)
\,\bigl|\det S_{j,t}\bigr|
~dx\;=\;
\lambda_i^t \sum_{j=1}^m 
\int_{T_{j,0}} \widetilde u_i^t\,w \,\bigl|\det S_{j,t}\bigr|~dx.
\end{equation}
Meanwhile, since for arbitrary $u\in E(\subset H^1_0(K^0))$, we have
\begin{equation}\label{eq:variational-formula-poly-v}
\int_{K^0} \nabla u\cdot \nabla v\,dx 
=\; 
\lambda \int_{K^0} u\,v\,dx, 
\quad \forall\,v\in H^1_0(K^0).
\end{equation}


Take $w :=u$ in \eqref{eq:perturbated-substitute-u-Kt-v} and $v:=\widetilde{u}_i^t$ in \eqref{eq:variational-formula-poly-v}. 
Then, we obtain
\begin{equation}\label{eq:perturbated-substitute-u-Kt}
    \sum_{j=1}^m 
\int_{T_{j,0}}
\bigl(S_{j,t}^{-\mathsf{T}} \,\nabla \widetilde u_i^t \bigr)\cdot 
\bigl(S_{j,t}^{-\mathsf{T}} \,\nabla u \bigr)
\,\bigl|\det S_{j,t}\bigr|
~dx\;=\;
\lambda_i^t \sum_{j=1}^m 
\int_{T_{j,0}} \widetilde u_i^t\,u \,\bigl|\det S_{j,t}\bigr|~dx
\end{equation}
and
\begin{equation}\label{eq:variational-formula-poly}
\int_{K^0} \nabla u\cdot \nabla \widetilde{u}_i^t\,dx 
=\; 
\lambda \int_{K^0} u\,\widetilde{u}_i^t\,dx, 
\quad \forall\,\widetilde{u}_i^t\in H^1_0(K^0).
\end{equation}

Subtracting  \eqref{eq:variational-formula-poly} from  \eqref{eq:perturbated-substitute-u-Kt} and dividing the result by $t$ yields the following expression on the left-hand side:
\begin{align}
    \begin{split}\label{eq:substract-left}
    \sum_{j=1}^{m}\frac{1}{t}&\left\{\left(\left(S_{j,t}^{-1}\right)\left(S_{j,t}^{-T}\right)\nabla\tilde u_i^t,\nabla u\right)_{T_{j,0}}\cdot|\det S_{j,t}|-\left(\nabla\tilde u_i^t,\nabla u\right)_{T_{j,0}}\right\}\\
    &=
    \sum_{j=1}^{m}\frac{1}{t}\left\{\left(\left(S_{j,t}^{-1}\right)\left(S_{j,t}^{-T}\right)\nabla\tilde u_i^t,\nabla u\right)_{T_{j,0}}\cdot|\det S_{j,t}|-\left(\nabla\tilde u_i^t,\nabla u\right)_{T_{j,0}}\cdot|\det S_{j,t}|\right\}\\
    &~~+
    \sum_{j=1}^{m}\frac{1}{t}\left\{\left(\nabla\tilde u_i^t,\nabla u\right)_{T_{j,0}}\cdot|\det S_{j,t}|-\left(\nabla\tilde u_i^t,\nabla u\right)_{T_{j,0}}\right\}\\
    &=
    \sum_{j=1}^{m}\left[\left(P_{j,t}\nabla\tilde u_i^t,\nabla u\right)_{T_{j,0}}\cdot|\det S_{j,t}|+\left(\nabla\tilde u_i^t,\nabla u\right)_{T_{j,0}}\cdot d_{j,t}\right].
    \end{split}
\end{align}
The expression on the right-hand side becomes
\begin{align}
    \begin{split}\label{eq:substract-right}
    \sum_{j=1}^{m}\frac{1}{t}&\left\{\lambda^t_i(\tilde u_i^t,u)_{T_{j,0}}\cdot|\det S_{j,t}|-\lambda(\tilde u_i^t,u)_{T_{j,0}}\right\}\\
    &=
    \sum_{j=1}^{m}\frac{1}{t}\bigr\{\lambda^t_i(\tilde u_i^t,u)_{T_{j,0}}\cdot|\det S_{j,t}|-\lambda(\tilde u_i^t,u)_{T_{j,0}}\cdot|\det S_{j,t}|\\
    &\;\;\;\;\;\;\;\;\;\;\;\;\;\;+\lambda(\tilde u_i^t,u)_{T_{j,0}}\cdot|\det S_{j,t}|-\lambda(\tilde u_i^t,u)_{T_{j,0}}\bigr\}\\
    &=
    \frac{\lambda^t_i–\lambda}{t}\sum_{j=1}^{m}\left\{(\tilde u_i^t,u)_{T_{j,0}}\cdot|\det S_{j,t}|\right\}+\sum_{j=1}^{m}\left\{\lambda(\tilde u_i^t,u)_{T_{j,0}}\cdot \frac{1}{t}\left(|\det S_{j,t}|-1\right)\right\}\\
    \end{split}
\end{align}
Combining \eqref{eq:substract-left} with \eqref{eq:substract-right} yields
\begin{equation}
    \tilde a_t(\tilde u_{i}^t,u)
    =
    D_t\lambda_i\cdot
    \tilde b_t(\tilde u_{i}^t,u).
\end{equation}

\end{proof}

\begin{Remark}
The computation method of
``shape difference quotient'' was proposed in our early paper \cite{arxiv.2305.14063}, where the discussion is mainly focusing on a triangle domain.
In this paper, our approach models the perturbation of polygonal domains via piecewise-affine maps, using the shape difference quotient to transform the problem into a small, algebraic 
eigenvalue problem. This contrasts with classical methods that rely on smooth velocity fields and analytical boundary-integral formulas (e.g., Hadamard's formula). Our technique is therefore highly compatible with the finite element method and computationally advantageous for the perturbations to polygonal domains.

\end{Remark}

Assume that $\mbox{dim}(\widetilde{E}_t)=N-n+1$, where $\widetilde E_t := \mathrm{span}\{\widetilde u_n^t,\dots,\widetilde u_N^t\}$. Such an assumption is expected to hold for a general purturbation and easy to validate for concrete cases.
Because $E$ and $\widetilde E_t $ have the same (finite) dimension, one may choose bases of each and obtain a finite‐dimensional matrix formulation:

\begin{The}\label{lem:Fte-basis-eigen}
Suppose $\dim \widetilde E_t = \dim E =: M = N-n+1$.  Let 
\[
\{\,\widetilde\phi_i\,\}_{i=n}^N\quad\text{and}\quad 
\{\,\phi_j\,\}_{j=n}^N
\]
be arbitrary bases of $\widetilde E_t$ and $E$, respectively.  Define the two $M\times M$ matrices
\begin{equation}\label{eq:def-Mt-Nt}
(M_t)_{ij} \;:=\; \widetilde a_t\bigl(\widetilde\phi_i,\,\phi_j\bigr),
\qquad
(N_t)_{ij} \;:=\; \widetilde b_t\bigl(\widetilde\phi_i,\,\phi_j\bigr),
\quad 
i,j=n,\dots,N.
\end{equation}
Then each difference quotient $D_t\lambda_i$ appears as the $(i-n+1)$‐th eigenvalue of the generalized matrix eigenproblem
\begin{equation}\label{eq:Mt-N-eig-prob}
M_t\,\sigma \;=\; \mu\,N_t\,\sigma.
\end{equation}
Moreover, if $\sigma_i = (s_{n,i},\dots,s_{N,i})^{\mathsf{T}}$ is an eigenvector corresponding to the $(i-n+1)$‐th eigenvalue of \eqref{eq:Mt-N-eig-prob}, then 
\[
\widetilde u_i^t 
= \sum_{k=n}^N s_{k,i}\,\widetilde\phi_k 
\;\in\;\widetilde E_t,\quad 
i=n,\dots,N,
\]
and hence the stabilized eigenfunction on the perturbated domain $K^t$ is 
\[
u_i^t 
= \widetilde u_i^t \circ \Phi_t^{-1}
= \sum_{k=n}^N s_{k,i}\,\bigl(\widetilde\phi_k \circ \Phi_t^{-1}\bigr),
\quad i=n,\dots,N.
\]
\end{The}

\begin{proof}
Since for each $i,j=n,\dots,N$,
\[
\widetilde a_t(\widetilde u_i^t,\,\phi_j) 
= D_t\lambda_i \;\widetilde b_t(\widetilde u_i^t,\,\phi_j),
\]
and because each $\widetilde u_i^t$ admits a unique expansion 
\[
\widetilde u_i^t 
= \sum_{k=n}^N s_{k,i}\,\widetilde\phi_k, 
\]
substituting into the variational identity and collecting coefficients immediately yields
\[
M_t\,\sigma_i \;=\;(D_t\lambda_i)\;N_t\,\sigma_i, 
\quad 
\sigma_i = (s_{n,i},\dots,s_{N,i})^{\mathsf{T}}.
\]
Since the $M$ eigenvectors $\{\sigma_i\}$ are independent, their associated eigenvalues $D_t\lambda_i$ appear in ascending order along the generalized spectrum of \eqref{eq:Mt-N-eig-prob}.  This completes the proof.
\end{proof}

\section{Algorithm}\label{section:algorithm}

Using Theorem~\ref{lem:Fte-basis-eigen}, we now present a practical algorithm to stably compute the eigenfunctions associated with clustered eigenvalues. The only assumption is that the difference quotients $D_t\lambda_i$ remain sufficiently separated, so that the small matrix eigenproblem \eqref{eq:Mt-N-eig-prob} does not itself have clustered eigenvalues.

\begin{breakablealgorithm}
    \label{alg:stabilized_computation} 
     \caption{}
    \KwData{
        \begin{itemize}
            \item Unperturbated polygonal domain $K^0 \subset \mathbb{R}^2$.
            \item Theoretically multiple Diricglet eigenvalues $\lambda_n^0, \dots, \lambda_N^0$ on $K^0$.
            \item Perturbation parameter $t > 0$ (small) and perturbation direction $e \in \mathbb{R}^{2k}$ defining the perturbated domain $K^t = K^{p_t}$ where $p_t = p_0 + te$.
        \end{itemize}
    }
    \KwResult{Stabilized approximate eigenfunctions $u_i^t$ of $\lambda_i^t$ corresponding to clustered Dirichlet eigenvalues $\lambda_n^t, \dots, \lambda_N^t$..}
    
    \BlankLine 
    
    \textbf{Preprocessing:}\\ 
        Define the triangulation $K^0 = \bigcup_{j=1}^m T_{j,0}$.\\
        For each triangle $T_{j,0}$ in $K^0$, determine the linear transformation $S_{j,t}$ that maps $T_{j,0}$ to the corresponding triangle $T_{j,t}$ in $K^t$. This defines the global transformation $\Phi_t: K^0 \to K^t$ such that $\Phi_t|_{T_{j,0}} = S_{j,t}|_{T_{j,0}}$.
    
    \BlankLine
    \textbf{Step 1: Solve Eigenproblem on perturbated Domain $K^0$}\\
        Use FEM to solve the Dirichlet eigenvalue problem $-\Delta u^0 = \lambda^0 u^0$ on $K^0$.\\
        Obtain a basis of the eigenspace $E \subset H_0^1(K^0)$ corresponding to $\lambda_n^0, \dots, \lambda_N^0$, denoted by $\{\phi_k\}_{k=n}^N$.\\

    \BlankLine
    \textbf{Step 2: Solve Eigenproblem on perturbated Domain $K^t$}\\
    Use FRM to solve the Dirichlet eigenvalue problem $-\Delta u^t = \lambda^t u^t$ on $K^t$, and obtain inaccurate eigenfunctions $\phi_n,\cdots,\phi_N$ corresponding to $\lambda_n^t, \dots, \lambda_N^t$.\\
    Let $\widetilde{E}_t:=\mbox{span}\{\tilde{\phi}_n,\cdots,\tilde{\phi}_N\}(\subset H_0^1(K^0))$.
    \BlankLine
    \textbf{Step 3: Construct Matrices $M_t$ and $N_t$}\\
        For $i,j = n, \dots, N$, compute the elements of matrices $M_t$ and $N_t$ of size $M \times M$:
        \begin{align*}
            (M_t)_{ij} := \tilde{a}_t(\tilde{\phi}_i, \phi_j) ,~~
            (N_t)_{ij} := \tilde{b}_t(\tilde{\phi}_i, \phi_j).
        \end{align*}

    \BlankLine
    \textbf{Step 4: Solve Generalized Matrix Eigenvalue Problem}\\
        Solve the $M \times M$ matrix eigenvalue problem $M_t \sigma = \mu N_t \sigma$ for eigenvalues $\mu_k$ and corresponding eigenvectors $\sigma_k$ for $k=n, \dots, N$.\\
        Order the eigenvalues $\mu_{(n)} \le \mu_{(n+1)} \le \dots \le \mu_{(N)}$. Let $\sigma_{(k)}$ be the eigenvector corresponding to $\mu_{(k)}$. These $\mu_{(i)}$ are the difference quotients $D_t\lambda_i$.

    \BlankLine
    \textbf{Step 5: Construct Stabilized Eigenfunctions $\tilde{u}_i^t$}\\
        For each $i = n, \dots, N$, let $\sigma_{(i)} = (s_{n,i}, s_{n+1,i}, \dots, s_{N,i})^T$ be the eigenvector obtained in Step 4 corresponding to $\mu_{(i)}$.\\
        The stabilized eigenfunction $u_i^t$ on $K^t$ is constructed as:
        $$ u_i^t := \sum_{k=n}^{N} s_{k,i} (\tilde{\phi}_k\circ \Phi_t^{-1}).$$
\end{breakablealgorithm}

\section{Numerical Example on a Rectangle}\label{sec:numerical_examples_rec}

This section numerically verifies the proposed algorithm on a perturbated rectangular domain 
\[
\Omega_{\varepsilon} = (0,1+\varepsilon)\times(0,1),
\]
with clustered eigenvalues $\{\lambda_2,\lambda_3\}$, focusing on antisymmetry preservation of the second $u_2$ and third $u_3$ eigenfunctions as $\varepsilon$ varies.

\subsection{Experimental Setup and Metrics}

Computations were performed using {\tt FEniCS} (v2019.1.0) with $P_1$ Lagrange elements on a fixed uniform mesh of size $h=1/64$. We investigated perturbation magnitudes $\varepsilon\in\{10^{-1},\,10^{-5},\,10^{-10}\}$, targeting the clustered pair $\{\lambda_2,\lambda_3\}$. Define the antisymmetry measure
\[
A_i \;:=\; \frac{\|\,u_i + u_i^\ast\|_{L^2(\Omega)}}{\|u_i\|_{L^2(\Omega)}}, 
\quad i=2,3,
\]
where $u_i^\ast$ is the reflection of $u_i$ about the theoretical axis (vertical line $x=(1+\varepsilon)/2$ for $u_2$, horizontal line $y=1/2$ for $u_3$). Perfect antisymmetry corresponds to $A_i=0$.

\subsection{Results and Discussion}

Table~\ref{tab:symmetry_results_varied_eps} presents the antisymmetry measures for $u_2$ and $u_3$ on $\Omega_{\varepsilon}$ with $h=1/64$, for each $\varepsilon$. As $\varepsilon$ decreases (tighter eigenvalue clustering), the standard FEM fails to preserve the theoretical antisymmetries (large $A_i$). In contrast, our proposed algorithm maintains $A_i\approx 7\times10^{-4}$ across all $\varepsilon$, demonstrating robustness under severe clustering.

\begin{table}[h]
  \centering
  \caption{Antisymmetry measures ($A_i$) and eigenvalue difference quotients ($D_\varepsilon\lambda_i$) for $u_2$ and $u_3$ on $\Omega_{\varepsilon}$ with $h=1/64$. Lower $A_i$ is better.}
  \label{tab:symmetry_results_varied_eps}
  \small
  \setlength{\tabcolsep}{4pt}
  \renewcommand{\arraystretch}{1.1}
  \begin{tabular}{@{}l l rr ccc@{}}
    \toprule
    $\varepsilon$ & $\lambda_3-\lambda_2$  & $D_\varepsilon\lambda_2$ & $D_\varepsilon\lambda_3$ & Method & $A_2$ & $A_3$ \\
    \midrule
    \multirow{2}{*}{$10^{-1}$} &
    \multirow{2}{*}{$6.85$} &
    \multirow{2}{*}{$-75.44$} &
    \multirow{2}{*}{$-18.86$}
      & Standard FEM      & $0.0044$  & $0.0049$  \\
      & & & & Proposed Alg.     & $\mathbf{0.0007}$  & $\mathbf{0.0007}$   \\
    \midrule
    \multirow{2}{*}{$10^{-5}$} &
    \multirow{2}{*}{$8.00\times 10^{-4}$} &
    \multirow{2}{*}{$-79.03$} &
    \multirow{2}{*}{$-19.76$}
      & Standard FEM      & $1.3993$  & $1.3997$  \\
      & & & & Proposed Alg.     & $\mathbf{0.0007}$  & $\mathbf{0.0007}$   \\
    \midrule
    \multirow{2}{*}{$10^{-10}$} &
    \multirow{2}{*}{$7.90\times 10^{-9}$} &
    \multirow{2}{*}{$-79.03$} &
    \multirow{2}{*}{$-19.76$}
      & Standard FEM      & $1.4140$  & $1.4144$  \\
      & & & & Proposed Alg.     & $\mathbf{0.0007}$  & $\mathbf{0.0007}$   \\
    \bottomrule
  \end{tabular}
\end{table}

\begin{figure}[!htbp]
  \centering
  \begin{subfigure}[b]{0.43\textwidth}
    \centering
    \includegraphics[width=\textwidth]{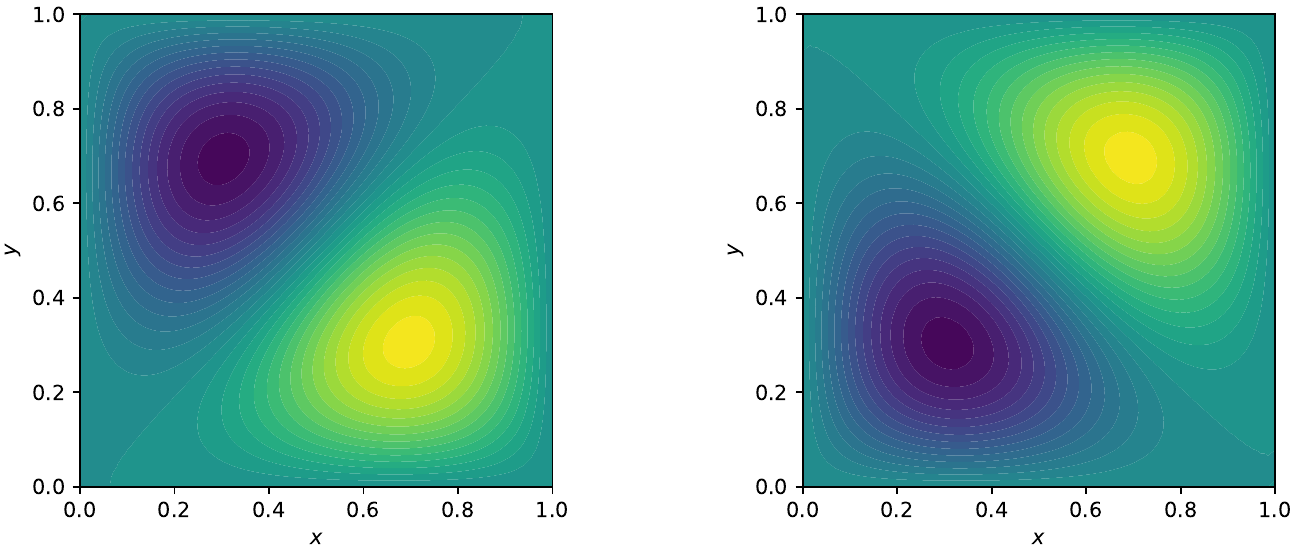}
    \caption{Eigenfunctions ($u_2$ left, $u_3$ right) via standard FEM.}
    \label{fig:fem_eigenfuncs_representative_ep5}
  \end{subfigure}
  \hfill
  \begin{subfigure}[b]{0.43\textwidth}
    \centering
    \includegraphics[width=\textwidth]{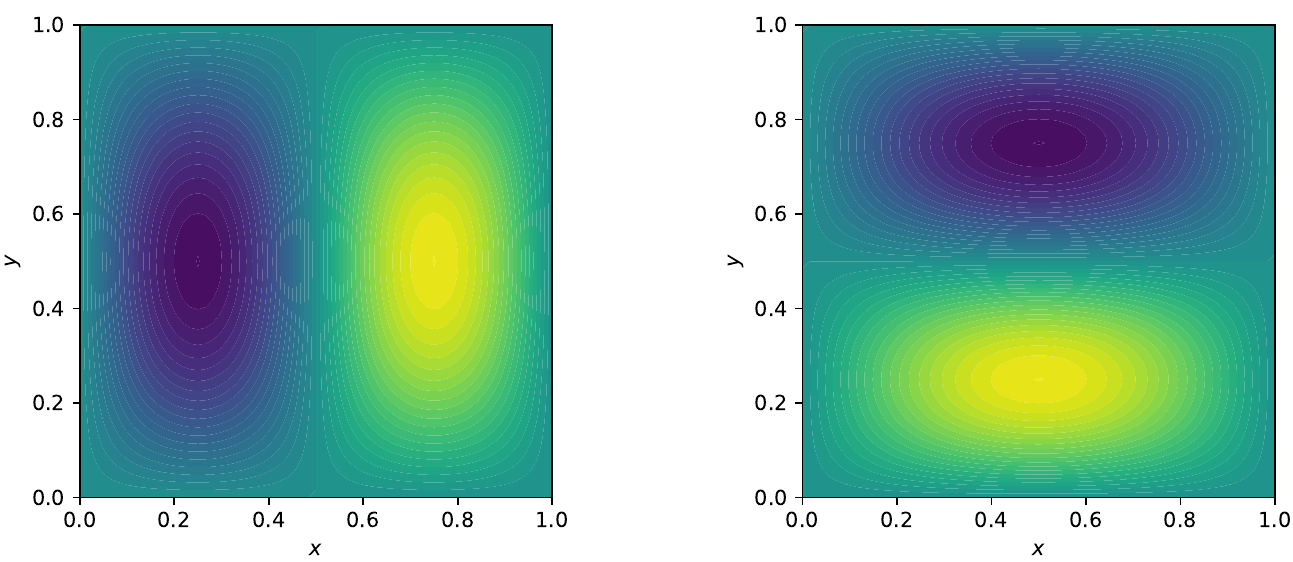}
    \caption{Eigenfunctions ($u_2$ left, $u_3$ right) via the proposed algorithm.}
    \label{fig:proposed_eigenfuncs_representative_ep5}
  \end{subfigure}

  \caption{\label{fig:fem_eigenfuncs_square} Eigenfunctions on rectangles for $\varepsilon=10^{-5}$.}
  \label{fig:comparison_eigenfuncs_ep5}
\end{figure}


Figure~\ref{fig:fem_eigenfuncs_square}(A) shows the second and third eigenfunctions computed by standard FEM on $\Omega(10^{-5})$, exhibiting severe loss of antisymmetry. In contrast, Figure~\ref{fig:fem_eigenfuncs_square}(B) displays the stabilized eigenfunctions obtained by the proposed algorithm on the same domain, where antisymmetry is restored.

\section{Numerical Example on a Perturbated Equilateral Triangle}

In this section, we apply our proposed algorithm to another classic problem involving eigenvalue clustering: the Laplace–Dirichlet problem on a nearly equilateral triangle.

It is well-known that for a perfect equilateral triangle, the second and third Dirichlet eigenvalues are degenerate \cite{mccartin2003eigenstructure}. We consider a triangle \(T(s,t)\) with the vertices \((0,0)\), \((1,0)\), and \((s,t)\), where the equilateral case corresponds to \((s,t) = \bigl(\tfrac{1}{2},\,\tfrac{\sqrt{3}}{2}\bigr)\).

 The multiplicity of the second eigenvalue over the equilateral triangle is a direct result of the symmetry of domain.
 By introducing a very small perturbation to a vertex, we break this symmetry and cause the single eigenvalue to split into a pair of tightly clustered eigenvalues. Standard numerical methods typically fail to resolve the corresponding eigenfunctions, often producing an arbitrary and unstable linear combination of the two modes. It is demonstrated that our algorithm can stably compute these eigenfunctions even with an extremely small perturbation of \(\epsilon = 10^{-6}\).

We tested four different perturbations by shifting the top vertex of the equilateral triangle:
\begin{itemize}
  \item [ \textbf{(A)} ] Horizontal shift right: \((s,t) = \bigl(\tfrac{1}{2} + \epsilon,\,\tfrac{\sqrt{3}}{2}\bigr)\).
  \item [ \textbf{(B)} ]  Horizontal shift left: \((s,t) = \bigl(\tfrac{1}{2} - \epsilon,\,\tfrac{\sqrt{3}}{2}\bigr)\).
  \item [ \textbf{(C)} ] Vertical shift up: \((s,t) = \bigl(\tfrac{1}{2},\,\tfrac{\sqrt{3}}{2} + \epsilon\bigr)\).
  \item [ \textbf{(D)} ] Vertical shift down: \((s,t) = \bigl(\tfrac{1}{2},\,\tfrac{\sqrt{3}}{2} - \epsilon\bigr)\).
\end{itemize}

Figure \ref{fig:all-triangles} displays the computed second \((u_{2}\), left column\() \) and third \((u_{3}\), right column\() \) eigenfunctions for each of these four perturbated domains. Qualitatively, the results show that the algorithm successfully separates the nearly degenerate modes into two distinct eigenfunctions whose structures reflect the nature of the symmetry-breaking perturbation.

For the horizontal perturbations (A and B), the algorithm resolves the eigenfunctions into modes that are almost antisymmetric about the triangle’s vertical axis. For the vertical perturbations (C and D), the computed eigenfunctions $u_3$ in C and $u_2$ in D show the perfect antisymmetry about the triangle’s vertical axis.

In all cases, despite the underlying eigenvalues being separated by only a minuscule amount, the proposed method robustly computes two smooth, well-defined, and orthogonal eigenfunctions. This result further validates the stability and accuracy of our approach for problems with clustered eigenvalues induced by domain perturbation.

\begin{figure}[!htbp]
  \centering
  \begin{subfigure}[b]{0.48\textwidth}
    \centering
    \includegraphics[width=\textwidth]{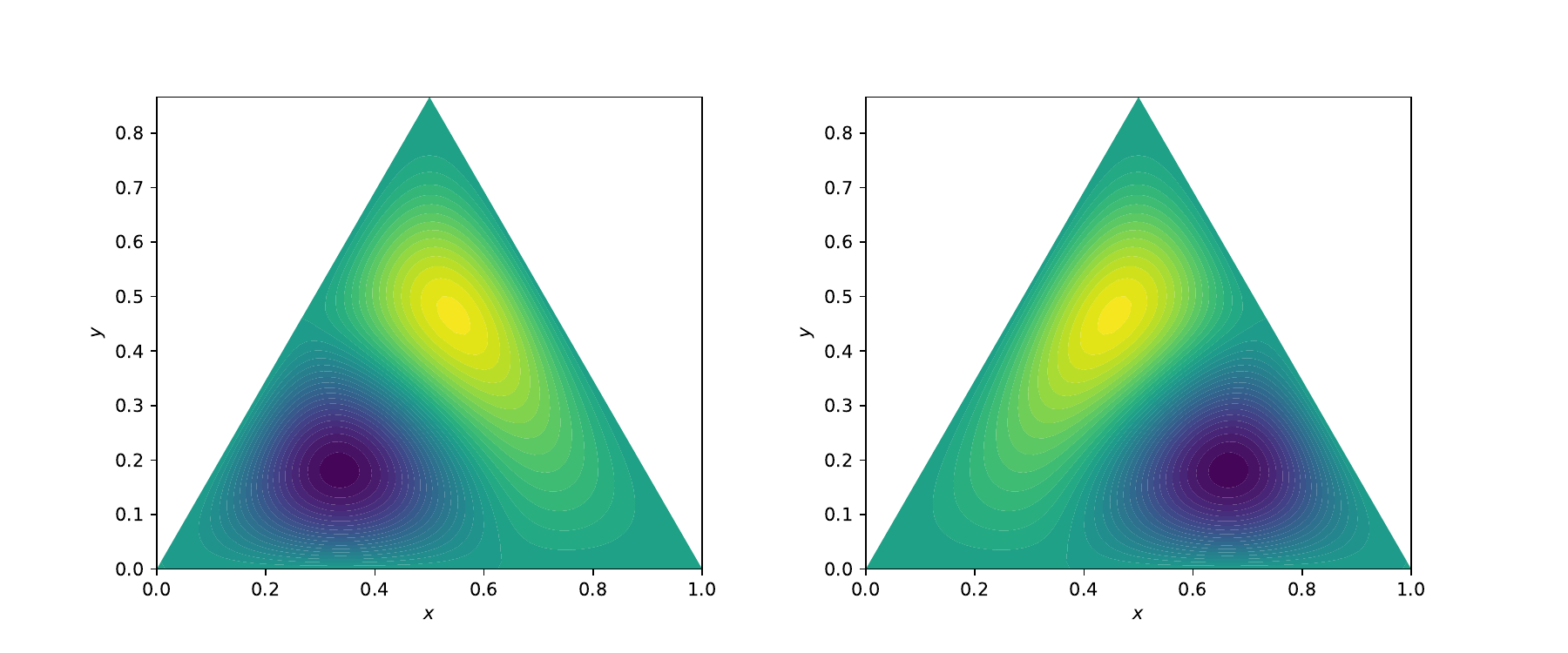}
    \caption{$(s,t)=(1/2+\varepsilon,\sqrt{3}/2)$.}
    \label{fig:+x-triangle}
  \end{subfigure}
  \hfill
  \begin{subfigure}[b]{0.48\textwidth}
    \centering
    \includegraphics[width=\textwidth]{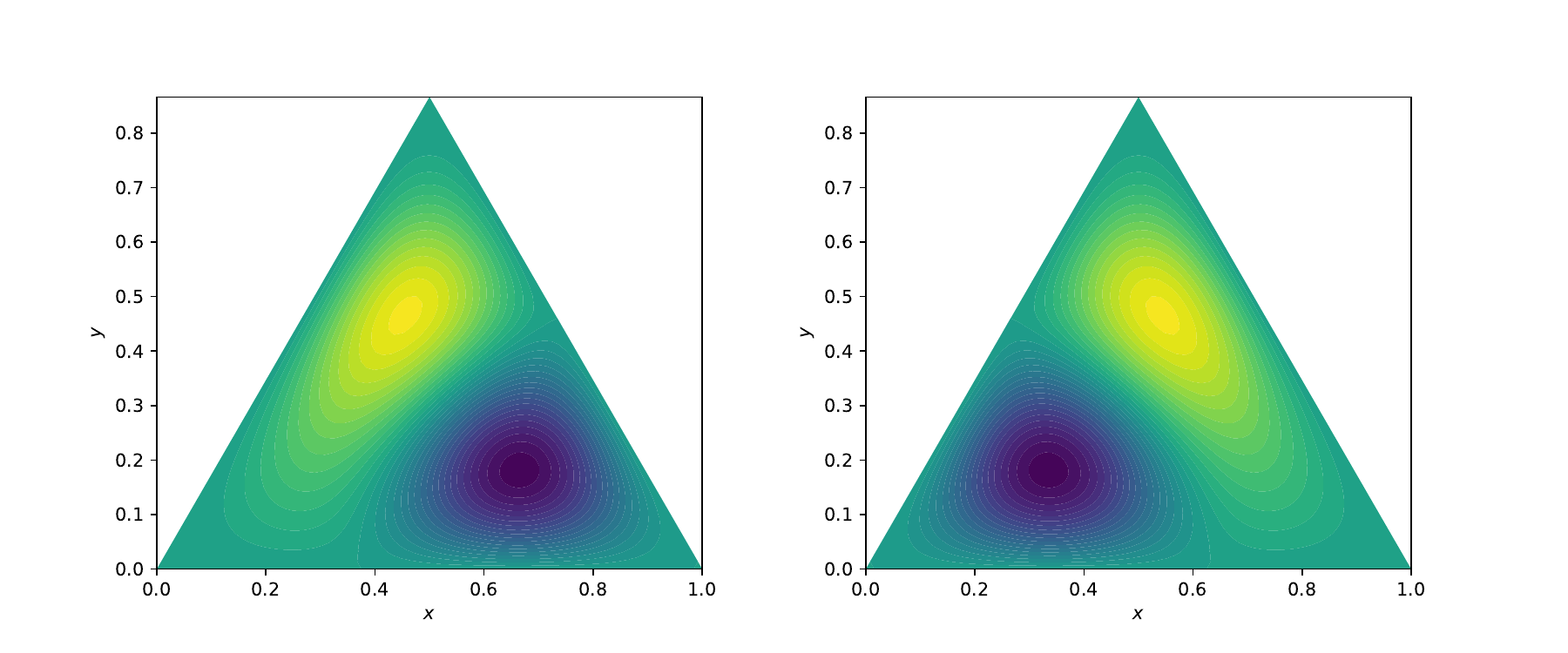}
    \caption{$(s,t)=(1/2-\varepsilon,\sqrt{3}/2)$.}
    \label{fig:-x-triangle}
  \end{subfigure}
  \vspace{2ex}

  \begin{subfigure}[b]{0.48\textwidth}
    \centering
    \includegraphics[width=\textwidth]{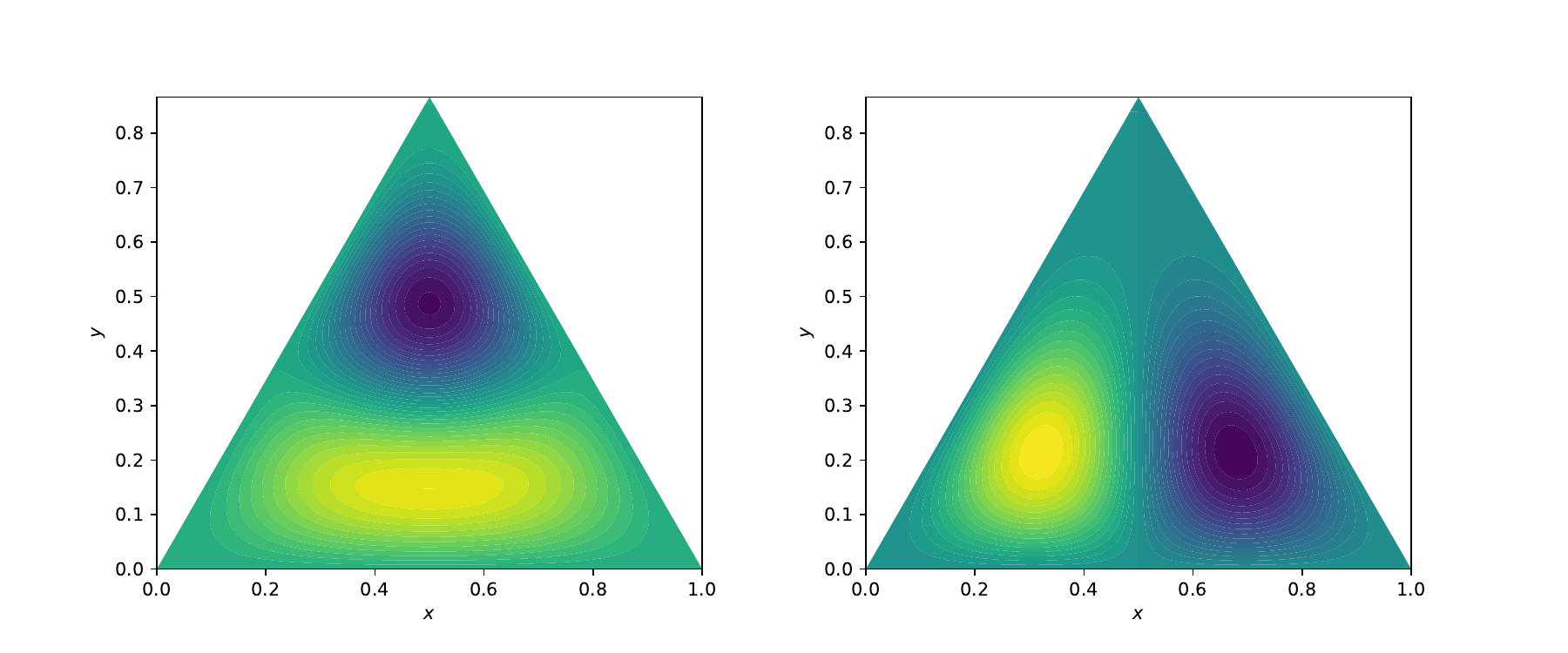}
    \caption{$(s,t)=(1/2,\sqrt{3}/2+\varepsilon)$.}
    \label{fig:y-triangle}
  \end{subfigure}
  \hfill
  \begin{subfigure}[b]{0.48\textwidth}
    \centering    \includegraphics[width=\textwidth]{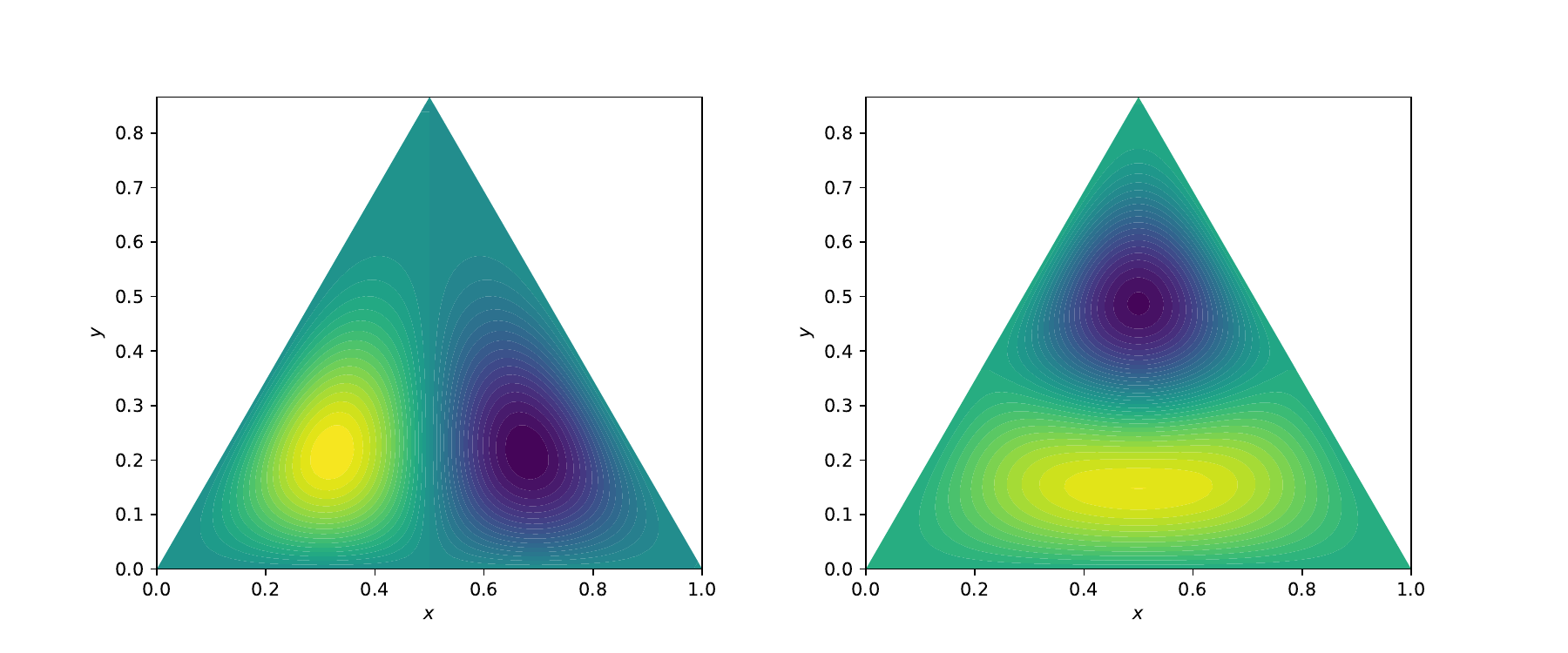}
    \caption{$(s,t)=(1/2,\sqrt{3}/2-\varepsilon)$.}
    \label{fig:-y-triangle}
  \end{subfigure}

  \caption{Eigenfunctions for $\varepsilon=10^{-6}$ ($u_2$ left, $u_3$ right)}

  \label{fig:all-triangles}
\end{figure}
\begin{Remark}
    For the perturbed equilateral triangle, the computed difference quotients of the eigenvalues, $\mu_2$ and $\mu_3$, maintain a large and numerically stable gap $\mu_3-\mu_2=75.76$, irrespective of the direction of perturbation. In contrast, the corresponding eigenvalues $\lambda_2$ and $\lambda_3$ become tightly clustered, as noted in Section 6. For the perturbation magnitude of $\epsilon = 10^{-6}$ used in our numerical examples, the gap $\lambda_3 - \lambda_2$ is on the order of $7.57\times 10^{-6}$. The existence of this stable, well-resolved gap in the difference quotients is the essential feature of our method that ensures the robust computation of distinct eigenfunctions for nearly degenerate modes.
\end{Remark}




\section{Conclusion and Future Work}\label{sec:conclusion}

We have proposed a novel algorithm that stabilizes the computation of Laplacian eigenfunctions under arbitrarily tight eigenvalue clustering caused by domain perturbations. By employing the shape difference quotient and reducing the ill‐posed problem to a small well‐separated generalized eigenproblem, we achieve accurate reconstruction of eigenfunctions even when standard FEM fails. Numerical results on perturbated rectangular domains confirm that our method preserves theoretical symmetries across a wide range of perturbation magnitudes.

Future research includes investigating rigorous a posteriori error estimates to quantify the accuracy of the stabilized eigenfunctions, where the recently developed guaranteed methods for computing eigenvalues and eigenfunctions will play an important role \cite{liu2024guaranteed}.
Another important direction is to develop stable eigenfunction computation techniques for clusters of eigenvalues that are not necessarily perturbations of a repeated eigenvalue—this being the basic assumption underlying our proposed method (see \eqref{eq:basic_assumption_on_eigenvalue}).

\section*{Acknowledgement}

Both authors are supported by Japan Society for the Promotion of Science. The first author is supported by JSPS KAKENHI Grant Number JP24KJ1170. The last author is supported by JSPS KAKENHI Grant Numbers JP20KK0306, JP22H00512, JP24K00538 and JP21H00998. This research is also supported by the bilateral joint research project (Japan-China) JPJSBP120237407.

\bibliographystyle{elsarticle-num} 
\bibliography{references}
						 
\end{document}